\documentclass[reqno]{amsart}
\usepackage{amssymb}
\usepackage{hyperref}
\usepackage{booktabs}
 \usepackage{multirow}

\newtheorem{theorem}{Theorem}[section]
\newtheorem{proposition}[theorem]{Proposition}

\theoremstyle{remark}
\newtheorem{remark}[theorem]{Remark}

\newtheorem*{note}{Note}

\numberwithin{equation}{section}

\begin{document}

\title[Bethe Ansatz and zeros of hypergeometric orthogonal polynomials]
{Solutions of convex Bethe Ansatz equations and the zeros of  (basic) hypergeometric orthogonal polynomials}

\author{J.F.  van Diejen}

\address{
Instituto de Matem\'atica y F\'{\i}sica, Universidad de Talca,
Casilla 747, Talca, Chile}

\email{diejen@inst-mat.utalca.cl}

\author{E. Emsiz}

\address{
Facultad de Matem\'aticas, Pontificia Universidad Cat\'olica de Chile,
Casilla 306, Correo 22, Santiago, Chile}
\email{eemsiz@mat.uc.cl}

\subjclass[2010]{Primary: 33D45; Secondary  26C10, 33C45, 81R12, 82B23}
\keywords{(basic) hypergeometric orthogonal polynomials, zeros of orthogonal polynomials, convex Bethe Ansatz equations}

\thanks{This work was supported in part by the {\em Fondo Nacional de Desarrollo
Cient\'{\i}fico y Tecnol\'ogico (FONDECYT)} Grants  \# 1170179 and \# 1181046.}

\date{May 2018}

\begin{abstract}
Via the solutions of systems of algebraic equations of Bethe Ansatz type, we arrive at bounds for
the zeros of orthogonal (basic) hypergeometric polynomials belonging to the Askey-Wilson, Wilson and continuous Hahn families.
\end{abstract}

\maketitle

%\tableofcontents

%%%%%%%%%%%%%%%%%%%%%%%%%%%%%%%%%%%%%%%%%%%%
%%%%%%%%%%%%%%% SECTION %%%%%%%%%%%%%%%%%%%%%%%
%%%%%%%%%%%%%%% SECTION %%%%%%%%%%%%%%%%%%%%%%%
%%%%%%%%%%%%%%% SECTION %%%%%%%%%%%%%%%%%%%%%%%
%%%%%%%%%%%%%%%%%%%%%%%%%%%%%%%%%%%%%%%%%%%%

\section{Introduction}
The study of the zeros of orthogonal polynomials has a rich history
 \cite{sze:orthogonal} stimulated, in particular, by  its relevance for the theory of numerical approximation
\cite{gau:orthogonal}. Of special interest are  the zeros of the classical families of hypergeometric orthogonal polynomials, which have  been 
fruitfully analyzed e.g. through  Sturm-Liouville theory and via
Stieltjes'  electrostatic interpretation 
\cite{sze:orthogonal,ahm-laf-mul:spacing,for-rog:electrostatics,gat:new,elb-laf-rod:zeros,gru:variations,ism:electrostatic,mar-mar-mar:electrostatic,jor-too:convexity,dri-jor:bounds,are-dimi-god-raf:inequalities,are-dimi-god-pas:zeros,sim:electrostatic}.
In this work we are mainly concerned with estimates for the locations of the zeros of some well-studied (basic) hypergeometric orthogonal polynomial families belonging to the ($q-$)Askey scheme \cite{ask-wil:some,koe-swa:hypergeometric}, while other relevant issues
concerning these zeros, such as e.g.
their dependence  on the parameters
\cite{sze:orthogonal,mul:properties,ism:classical}, interlacing properties  \cite{sze:orthogonal,dri:note,han-van:interlacing,goc-jor-rag-swa:interlacing}, or their asymptotical behavior \cite{sze:orthogonal,sim:fine,dea-huy-kui:asymptotic} will  not be addressed.

A classic bound for the locations of the zeros
\begin{equation}
\pi >\xi^{(n)}_1>\xi^{(n)}_2\cdots >\xi^{(n)}_{n-1}>\xi^{(n)}_n>0
\end{equation}
of the Jacobi polynomial $P_n^{(\alpha ,\beta)}(\cos (\xi))$ \cite[Ch. 9.8]{koe-swa:hypergeometric}  with $-\frac{1}{2} \leq \alpha ,\beta \leq \frac{1}{2}$
is provided by  Buell's estimate \cite[Eqs. (6.3.5), (6.3.7)]{sze:orthogonal}:
\begin{equation}\label{buell}
 \frac{\bigl( n+1-j+\frac{1}{2}(\alpha +\beta-1)\bigr) \pi}{n+\frac{1}{2}(\alpha +\beta+1)}  \leq \xi^{(n)}_j \leq \frac{(n+1-j)\pi }{n+\frac{1}{2}(\alpha +\beta+1)}  \qquad (j=1,\ldots ,n)
\end{equation}
(where all inequalities are actually strict unless $\alpha^2=\beta^2=\frac{1}{4}$).
For $\alpha=\beta= \frac{1}{2}$ the estimate in Eq. \eqref{buell} becomes  exact; indeed,
the Jacobi polynomial $P_n^{(\alpha ,\beta)}(\cos (\xi))$  degenerates in this situation to the Chebyshev polynomial of the second kind $U_n\bigl(\cos(\xi)\bigr)=\frac{\sin \bigl((n+1)\xi\bigr)}{\sin(\xi)}$.

Below we will derive a similar estimate for the corresponding zeros of the Askey-Wilson polynomial
$p_n(\cos (\xi);a,b,c,d; q)$ \cite{ask-wil:some},
\cite[Ch. 14.1]{koe-swa:hypergeometric} with parameters in the domain $-1 <  a,b,c,d, q<1$:
\begin{subequations}
\begin{equation}\label{aw-zeros:a}
 \frac{(n+1-j) \pi}{k_-^{(n)}(a,b,c,d ; q)}  \leq \xi^{(n)}_j \leq  \frac{(n+1-j) \pi}{k_+^{(n)}(a,b,c,d ; q)}  \qquad (j=1,\ldots ,n),
\end{equation}
where
\begin{align}\label{aw-zeros:b}
k_{\pm}^{(n)}& (a,b,c,d ; q)= (n-1)  \Biggl(\frac{1-|q|}{1+ |q|}\Biggr)^{\pm 1}+ \\
&\frac{1}{2}\left( \Biggl(\frac{1-|a|}{1+ |a|}\Biggr)^{\pm 1}+ \Biggl(\frac{1-|b|}{1+ |b|}\Biggr)^{\pm 1}+ \Biggl(\frac{1-|c|}{1+ |c|}\Biggr)^{\pm 1}+ \Biggl(\frac{1-|d|}{1+ |d|}\Biggr)^{\pm 1} \right) .
\nonumber
\end{align}
\end{subequations}
The estimate in Eqs. \eqref{aw-zeros:a}, \eqref{aw-zeros:b} becomes exact for vanishing parameters $a,b,c,d$ and $q$, and again
$p_n(\cos (\xi);0,0,0,0 ; 0)=U_n\bigl(\cos(\xi)\bigr)$ in this situation. Notice however that our formula renders only a trivial lower bound (viz. zero) if one of the parameters tends to $1$ in absolute value, and that---especially for the larger roots---the upper bound is nontrivial (i.e. smaller than $\pi$) only for parameter values
sufficiently close to $0$. In particular, the Askey-Wilson estimate in Eqs. \eqref{aw-zeros:a}, \eqref{aw-zeros:b} merely produces trivial bounds on the zeros of the
Jacobi polynomial via a direct application of
the well-known  degeneration formula (cf. e.g. \cite[Ch. 14.10]{koe-swa:hypergeometric}):
\begin{equation*}
\lim_{q\to 1}
c_n\, p_n(\cos(\xi); q^{\frac{\alpha}{2}+\frac{1}{4}}, q^{\frac{\alpha}{2}+\frac{3}{4}}, -q^{\frac{\beta}{2}+\frac{1}{4}}, -q^{\frac{\beta}{2}+\frac{3}{4}} ; q)=P_n^{(\alpha ,\beta)}(\cos (\xi))
\end{equation*}
(where---employing standard notations for the $q$-shifted factorials---the normalization
factor is given explicitly by  $c_n:= q^{(\frac{\alpha}{2}+\frac{1}{4})n} / (q,-q^{\frac{1}{2}(\alpha+\beta+1)}, -q^{\frac{1}{2}(\alpha+\beta+2)};q)_n$).

Many fundamental properties of the Askey-Wilson polynomials were first presented in the seminal memoir  \cite{ask-wil:some}.
The zeros were investigated at an early stage by L. Chihara (for $q$ integral and $a,b,c,d$ rational) in connection with the (non)existence of certain perfect codes
\cite{cih:zeros}. More recently it was observed that the locations of the zeros under consideration  are determined by an algebraic system of Bethe Ansatz equations  \cite{ism-lin-roa:bethe,die:equilibrium,oda-sas:equilibrium,bih-cal:properties:b} stemming from the celebrated second-order $q$-difference equation
satisfied by the Askey-Wilson polynomials\footnote{Whereas a classical result of Bochner characterizes the Jacobi polynomials  as ``the most general orthogonal family satisfying a linear homogeneous second-order differential equation'', the Askey-Wilson polynomials are known to constitute ``the most general orthogonal family satisfying a linear homogeneous second-order $q$-difference equation'' \cite{gru-hai:q-version}.}
 \cite[Ch. 16.5]{ism:classical}. Remarkably, this algebraic system turns out to be closely related to a family of Bethe Ansatz equations that emerge when diagonalizing $q$-boson models with open-end boundary interactions \cite{li-wan:exact,die-ems:orthogonality,die-ems-zur:completeness}.

We will consider two ample classes of (generalized) Bethe Ansatz type equations that are referred to as Bethe systems of type A and of type B. For special parameter choices, these Bethe equations of type A and of type B arise when diagonalizing quantum integrable particle models with periodic boundary conditions and with open-end boundary conditions, respectively.
Both types of systems manifest themselves here in three versions: rational $(r)$, hyperbolic $(h)$ and trigonometric $(t)$.  We will solve the corresponding  Bethe systems by means of a powerful technique going back to Yang and Yang
\cite{yan-yan:thermodynamics,mat:many-body,gau:bethe,kor-bog-ize:quantum}, which provides the Bethe roots in terms of the global minima of an associated family of strictly convex Morse functions and automatically produces bounds estimating the locations of these roots.
The bounds in Eqs. \eqref{aw-zeros:a}, \eqref{aw-zeros:b} for zeros of the Askey-Wilson polynomials  then follow
via a parameter-specialization of  the trigonometric Bethe systems of type B. Similarly, the corresponding rational
Bethe equations of type B give rise to lower bounds for the zeros of the 
Wilson polynomials \cite{wil:some}. Finally,  lower bounds for the zeros of the symmetric continuous Hahn polynomials \cite{ask-wil:set} are retrieved via the rational Bethe systems of type A.

The material  is organized as follows. In Sections \ref{sec2} and \ref{sec3} the Bethe systems of types A and B are exhibited in their general form and solved using the techniques of Yang and Yang \cite{yan-yan:thermodynamics}.
In Section \ref{sec4} we specialize the parameters in the  Bethe system of type B so as to achieve the bounds for the zeros of the Wilson polynomials (rational version)
and
the Askey-Wilson polynomials  (trigonometric version), respectively.
In Section \ref{sec5}, an analogous  parameter specialization of the rational Bethe system of type A then entails the lower bounds for the zeros of the symmetric continuous Hahn polynomials.

\begin{note}
Throughout it will be implicitly assumed that empty products are equal to $1$ and that empty sums are equal to $0$. We will also freely employ standard notations for (basic) hypergeometric series and $(q$-)shifted factorials in accordance with the conventions in Ref. \cite{koe-swa:hypergeometric}.
\end{note}

\section{Convex Bethe systems of type A}\label{sec2}
The idea of the Bethe Ansatz method is  to convert the spectral problem for amenable quantum integrable particle models into an algebraic problem: the spectrum is computed through a complete set of solutions to an auxiliary system of algebraic equations \cite{mat:many-body,gau:bethe,tak:integrable,kor-bog-ize:quantum}.
The combinatorics of such Bethe Ansatz solutions was investigated recently in Ref. \cite{koz-skl:combinatorics}.
Here we will consider a rather wide class of (in general transcendental)
Bethe type equations that are {\em convex} in the sense that they can be solved in terms of the critical points of an associated family of strictly convex Morse functions using the approach of Yang and Yang \cite{yan-yan:thermodynamics,mat:many-body,gau:bethe,kor-bog-ize:quantum}.
Special instances of the type A systems  in this section have appeared in the literature in connection with the spectral problems of exactly solvable quantum particle models with periodic boundary conditions. Specifically, the Bethe Ansatz equations governing the spectral problems of the periodic Lieb-Liniger quantum nonlinear Schr\"odinger equation
\cite{lie-lin:exact,yan-yan:thermodynamics,mat:many-body,gau:bethe,dor:orthogonality,kor-bog-ize:quantum}
and its lattice discretization due to Izergin and Korepin  \cite{ize-kor:lattice,dor:orthogonality} (cf. also \cite[Ch. VIII.3]{kor-bog-ize:quantum})
correspond to rational systems of type A. Trigonometric/hyperbolic systems of type A  arise in turn as Bethe Ansatz equations for the 
 $q$-boson model \cite{bog-bul:q-deformed,bog-ize-kit:correlation,tsi:quantum,die:diagonalization,kor:cylindric} and for the
lattice quantum sine-Gordon equation \cite[Ch. VIII.5]{kor-bog-ize:quantum}. The well-studied periodic Heisenberg XXX and XXZ spin chains also give rise to rational and trigonometric/hyperbolic Bethe Ansatz equations of type A  \cite{mat:many-body,gau:bethe,bax:exactly,kor-bog-ize:quantum}, but
 these spin models correspond to parameter values that do not belong to the convex regime considered here. The results in this section therefore do not apply directly to such models
and alternative techniques have been incorporated to analyze the corresponding Bethe Ansatz equations in these situations,
cf. e.g. Refs. \cite{muk-tar-var:bethe} and \cite{koz:condensation} and references therein.

\subsection{Bethe equations}\label{BE-A}
Given $u\in \{ r,h,t\}$ and 
\begin{equation}\label{S}
\text{s}(x)=\text{s}^{(u)}(x):=
\begin{cases}
\frac{x}{2} &\text{if}\ u=r, \\
\sinh(\frac{x}{2}) &\text{if}\ u=h, \\
\sin(\frac{x}{2}) &\text{if}\ u=t ,
\end{cases}
\end{equation}
the Bethe system of type A is defined by $n$  equations in the variables $\xi_1,\ldots ,\xi_n$ of the form
\begin{equation}\label{BAE-A}
\boxed{e^{i\alpha  \xi_j }
= e^{2\pi i \beta}
\prod_{1\le k \le K}
 \frac{\text{s}  (ia_k+\xi_j)}{\text{s}(ia_k-\xi_j)}
\prod_{\substack{1\le j^\prime \le n, \, j^\prime\neq j \\ 1 \le l \le L}}
 \frac{\text{s}(ib_{l}  + \xi_j-\xi_{j^\prime})}{\text{s}(ib_{l}- \xi_j + \xi_{j^\prime} )} }
\end{equation}
($j=1,\dots, n$). Here $\alpha, \beta$, $a_1,\ldots ,a_K$ and $b_1,\ldots ,b_L$ refer to a  choice of $K+L+2$  parameters.
Throughout  this section we 
%pick $L>0$ to ensure that the system is irreducible, and we 
assume (unless explicitly stated otherwise) that
\begin{equation}\label{convex-A}
\boxed{\alpha\in (0,\infty ), \quad \beta\in [0,1), \quad  a_k, b_l  \in \begin{cases}  (0,\infty) &\text{if}\ u=r \\ (0,\pi )&\text{if}\ u=h \\  (0,\infty )\cup\big( (0,\infty)+i\pi \bigr) &\text{if}\ u=t  \end{cases}}
\end{equation}
($k=1,\ldots ,K$, $l=1,\ldots, L$). As will be seen shortly, these restrictions on the parameters guarantee that the roots of our Bethe system are governed by a convex (Yang-Yang) Morse function.

\subsection{Bethe roots}\label{BR-A}
For any $\mu=(\mu_1,\ldots,\mu_n)$ belonging to
\begin{equation}\label{weights-A}
\Lambda_{\text{A}}:=   
\{ \mu\in\mathbb{Z}^n \mid   \mu_1>\mu_2> \cdots > \mu_n\} ,
\end{equation}
let us define the following Morse function in the variable $\boldsymbol{\xi}=(\xi_1,\ldots ,\xi_n)\in\mathbb{R}^n$:
\begin{align}\label{Morse-A}
V^{(u)}_{\text{A},\mu} (\boldsymbol{\xi};\alpha ,\beta):=&
\sum_{1\leq j\leq n}  \Biggl(
{ \frac{1}{2} }\alpha \xi_j^2-2\pi (\mu_j+\beta )\xi_j 
+  \sum_{1\leq k\leq K} \int^{\xi_j}_0 v_{a_k}(x)\, \text{d}x\Biggr) \\
&+  \sum_{\substack{1\le j < j^\prime \le n \\ 1\leq l\leq L}}
     \int_0^{\xi_j-\xi_{j^\prime}} v_{b_l}(x)\, \text{d}x
  ,\nonumber
\end{align}
with 
\begin{equation}\label{v}
v_a (x) = v_a^{(u)}(x) :=
\begin{cases}
\int_0^x \frac{2a}{a^2+y^2}  \text{d}y=2\arctan(\frac{x}{a}) & \text{if $u=r$},\\
\int_0^x \frac{ \sin(a)}{\cosh(y)-\cos(a)} \text{d}y & \text{if $u=h$},\\
\int_0^x \frac{ \sinh(a)}{\cosh(a)-\cos(y)} \text{d}y& \text{if $u=t$}.
\end{cases}
\end{equation}
In the spirit of Yang and Yang \cite{yan-yan:thermodynamics}, this Morse function is designed in such a way that the equation for its critical points (cf. Eq. \eqref{CE-A} below) reproduces our Bethe equation \eqref{BR-A} upon exponentiation (cf. the proof of Proposition \ref{Bethe-roots-A:prp} below). The parameter restrictions in Eq.  \eqref{convex-A} now guarantee that
$V^{(u)}_{\text{A},\mu} (\boldsymbol{\xi};\alpha ,\beta )$   \eqref{Morse-A} is smooth and that
$V^{(u)}_{\text{A},\mu} (\boldsymbol{\xi};\alpha ,\beta )\to \infty$
when $|\boldsymbol{\xi}|\to\infty$,  so the function  in question possesses  a global minimum in $\mathbb{R}^n$.
Notice in this connection that the contributions from the integrals in Eq. \eqref{Morse-A} are nonnegative for all  $\boldsymbol{\xi}\in\mathbb{R}^n$ (because $v_a(x)$ \eqref{v} is odd and increasing in $x$ for the relevant parameter values) and that for $|\boldsymbol{\xi}|\to\infty$ the nonnegative quadratic terms up front
dominate  possibly negative contributions stemming from the linear terms.

The minimum in question is unique by convexity. Indeed, the Hessian
\begin{align}\label{Hesse-A}
&H^{\text{A}}_{j,j^\prime} (\boldsymbol{\xi}):=\partial_{\xi_j}\partial_{\xi_{j^\prime}} V^{(u)}_{\text{A},\mu} (\boldsymbol{\xi};\alpha ,\beta ) \\
&={\footnotesize
\begin{cases}
\alpha + \sum_{1\le k \le K}v_{a_k}'(\xi_j)  
+
 \sum_{\substack{\ell\neq j\\ 1\le l \le L}} v_{b_l}'(\xi_j-\xi_\ell) & \text{if $ j^\prime=j$,}\\
\sum_{1\le l \le L}  v_{b_l}'(\xi_{j^\prime}-\xi_j) & \text{if $j^\prime\neq j$,}\\
\end{cases} }\nonumber
\end{align}
is manifestly positive definite:
\begin{eqnarray*}
\lefteqn{\sum_{1\leq j,j^\prime\leq n}  x_j x_{j^\prime}H^{\text{A}}_{j,j^\prime}  (\boldsymbol{\xi})  
= }&& \\
&& \sum_{1\leq j\leq n} \Bigl( \alpha
+   \sum_{1\le k \le K}v_{a_k}'(\xi_j)  
 \Bigr) x_j^2 
+  \sum_{\substack{1\leq j<j^\prime \leq n \\ 1\le l \le L}} 
v_{b_l}'(\xi_j -\xi_{j^\prime})(x_j-x_{j^\prime})^2
\end{eqnarray*}
(since the derivatives $v^\prime_{a_k}(x) $ and $ v^\prime_{b_l}(x)   $ 
are positive for our parameter regime).

The upshot is that for any $\mu\in\Lambda_{\text{A}}$ \eqref{weights-A} the critical
equation $\nabla_{\boldsymbol{\xi}} V^{(u)}_{\text{A},\mu} (\boldsymbol{\xi};\alpha ,\beta)=0$, which is given explicitly by the following transcendental system
\begin{equation}\label{CE-A}
\alpha\xi_j +
 \sum_{1 \le k \le K} 
 v_{a_k}(\xi_j)
+\sum_{\substack{1\le j^\prime \le n,\,  j^\prime \neq j \\ 1 \le l\le L}}   v_{b_l}(\xi_j-\xi_{j^\prime})
=2\pi (\mu_j+\beta)
\end{equation}
($j=1,\dots, n$), has a unique solution $\boldsymbol{\xi}=\boldsymbol{\xi}^{(u)}_{\text{A},\mu}$ consisting of the global minimum
of the strictly convex Morse function $V^{(u)}_{\text{A},\mu} (\boldsymbol{\xi};\alpha ,\beta )$ \eqref{Morse-A}.

\begin{proposition}[Bethe roots of type A]\label{Bethe-roots-A:prp}
(i) For the parameter regime in Eq. \eqref{convex-A} and
any $\mu\in\Lambda_{\text{A}}$ \eqref{weights-A}, the unique global minimum $\boldsymbol{\xi}=\boldsymbol{\xi}^{(u)}_{\text{A},\mu}$ of
$V^{(u)}_{\text{A},\mu} (\boldsymbol{\xi};\alpha ,\beta)$ \eqref{Morse-A} produces a solution of the
 Bethe system of type A \eqref{BAE-A}.

(ii) The assignment $\mu\to \boldsymbol{\xi}^{(u)}_{\text{A},\mu}$, $\mu\in\Lambda_{\text{A}}$ is injective and  depends smoothly on the parameters
 \eqref{convex-A}.

\end{proposition}
\begin{proof}
(i) Since for any $x\in\mathbb{R}$ and $a=a_k$, $b_l$ subject to the restrictions in Eq. \eqref{convex-A}:
\begin{equation}
e^{-i v_a^{(u)}(x)} =\frac{\text{s}^{(u)}(ia+x)}{\text{s}^{(u)}(ia-x)},
\end{equation}
it is clear from Eq. \eqref{CE-A}---upon multiplying by $i$ and exponentiating both sides---that the critical point $\boldsymbol{\xi}=\boldsymbol{\xi}^{(u)}_{\text{A},\mu}$ solves the Bethe equations \eqref{BAE-A}.

(ii) That the assignment $\mu\to \boldsymbol{\xi}^{(u)}_{\text{A},\mu}$, $\mu\in\Lambda_{\text{A}}$ defines an injection is also immediate from the system in Eq. \eqref{CE-A}, while the smoothness in the parameters  \eqref{convex-A} follows from it by the implicit function theorem. Indeed, our system is manifestly smooth in these parameters and its Jacobian amounts to the Hessian \eqref{Hesse-A} (and is thus invertible).
\end{proof}

\subsection{Bethe bounds}\label{BB-A}
It follows from the system in Eq. \eqref{CE-A} that the global minimum
$\boldsymbol{\xi}^{(u)}_{\text{A},\mu}$ of $V^{(u)}_{\text{A},\mu} (\boldsymbol{\xi};\alpha ,\beta)$ \eqref{Morse-A} with  $\mu\in\Lambda_{\text{A}}$ \eqref{weights-A}
belongs to the open wedge
\begin{equation}\label{wedge-A}
\mathbb{A}_{\text{A}}:= \{ \boldsymbol{\xi}\in\mathbb{R}^n\mid \xi_1 >\xi_2>\cdots >\xi_n \} .
\end{equation}
Indeed, by subtracting the $j^\prime$th equation from the $j$th equation we see that
\begin{align}\label{CED-A}
\alpha (\xi_j -\xi_{j^\prime}) &+
 \sum_{1 \le k \le K} 
\bigl(  v_{a_k}(\xi_j)- v_{a_k}(\xi_{j^\prime}) \bigr)  \\
&+\sum_{\substack{1\le j^{\prime\prime} \le n\\ 1 \le l\le L}}  \bigl(  v_{b_l}(\xi_j-\xi_{j^{\prime\prime}})-v_{b_l}(\xi_{j^\prime}-\xi_{j^{\prime\prime}})\bigr)
=2\pi (\mu_j-\mu_{j^\prime}) .\nonumber
\end{align}
Since $v_{a_k}(x)$ and $v_{b_l}(x)$ are monotonously increasing for our parameter regime, it is manifest from Eq. \eqref{CED-A} that
$\xi_j > \xi_{j^\prime}$ if $\mu_j > \mu_{j^\prime}$ (because otherwise the LHS would be $\leq 0$ while the RHS is $>0$). By refining this analysis somewhat further, one arrives at the following bounds on the gaps between $\xi_j$ and $\xi_{j^\prime}$ at 
$\boldsymbol{\xi}=\boldsymbol{\xi}^{(u)}_{\text{A},\mu}$.

\begin{proposition}[Bethe bounds of type A]\label{Bethe-bounds-A:prp}
For the parameter regime in Eq. \eqref{convex-A} and
any $\mu\in\Lambda_{\text{A}}$ \eqref{weights-A}, one has that at the global minimum $\boldsymbol{\xi}=\boldsymbol{\xi}^{(u)}_{\text{A},\mu}$ of
$V^{(u)}_{\text{A},\mu} (\boldsymbol{\xi};\alpha ,\beta)$ \eqref{Morse-A}:
\begin{equation}\label{bounds-A}
\frac{2\pi(\mu_j-\mu_{j^\prime})}{\alpha+\kappa^{(u)}_-} \le \xi_j-\xi_{j^\prime} \le \frac{2\pi(\mu_j-\mu_{j^\prime})}{\alpha+\kappa^{(u)}_+}
\end{equation}
for $1\leq j<j^\prime\leq n$, where
\begin{equation*}
\kappa^{(u)}_-  := \begin{cases}
2\sum_{1\leq k\leq K} a_k^{-1}+2n \sum_{1\leq l\leq L}   b_l^{-1}&\text{if}\ u=r, \\
\sum_{1\leq k\leq K} \cot \bigl(\frac{1}{2} a_k\bigr)+n \sum_{1\leq l\leq L}   \cot \bigl(\frac{1}{2}b_l\bigr) &\text{if}\ u=h, \\
\sum_{1\leq k\leq K} \coth \bigl(\frac{1}{2} \emph{Re} (a_k)\bigr)+n \sum_{1\leq l\leq L}   \coth \bigl(\frac{1}{2}\emph{Re} (b_l)\bigr) &\text{if}\ u=t ,\end{cases}
\end{equation*}
and
\begin{equation*}
\kappa^{(u)}_+ := \begin{cases}
0&\text{if}\ u=r\ \text{or}\ u=h, \\
\sum_{1\leq k\leq K} \tanh \bigl(\frac{1}{2} \emph{Re}(a_k)\bigr)+n \sum_{1\leq l\leq L}   \tanh \bigl(\frac{1}{2}\emph{Re}(b_l)\bigr) &\text{if}\ u=t .\end{cases}
\end{equation*}
\end{proposition}
\begin{proof}
When $a=a_k$ or $b_l$ from Eq. \eqref{convex-A}, the derivative of $v_a (x)=v^{(u)}_a (x) $ \eqref{v}  remains bounded:
\begin{equation*}
v_a^\prime (x)\in \begin{cases}
\left[0,\frac{2}{a}\right]   &\text{if}\  u=r ,\\
\left[0,\cot(\frac{a}{2})\right]   &\text{if}\  u=h, \\
\left[\tanh\big(\frac{1}{2}\text{Re}(a)\bigr),\coth\bigl(\frac{1}{2}\text{Re}(a)\bigr)\right]   &\text{if}\  u=t 
\end{cases}
\end{equation*}
(for any $x\in\mathbb{R}$).
Hence, if $1\leq j<j^\prime\leq n$ (so $\mu_j>\mu_{j^\prime}$ and  $\xi_j>\xi_{j^\prime}$)  Eq. \eqref{CED-A} implies (via the mean value theorem) that 
\begin{equation*}
  (\alpha +\kappa^{(u)}_+) (\xi_j-\xi_{j^\prime})     \leq    2\pi  (\mu_j-\mu_{j^\prime})          \leq  (\alpha +\kappa^{(u)}_-) (\xi_j-\xi_{j^\prime}) .
\end{equation*}
\end{proof}

\section{Convex Bethe systems of type B}\label{sec3}
In the case of quantum particle models with open-end boundary conditions rather than the more conventional periodic boundary conditions, the form of the Bethe Ansatz equations
is known to undergo some structural modifications \cite{gau:boundary,gau:bethe,skl:boundary}. Here we refer to the latter Bethe Ansatz equations as systems of type B, and we will again restrict attention
to a relatively wide class of equations in the convex regime. The mathematics  of trigonometric Bethe systems of type B 
was  addressed in
Refs. \cite{ism-lin-roa:bethe} (cf. also \cite[Ch. 16.5]{ism:classical}) and
\cite{bas:q-deformed} through $q$-difference Sturm-Liouville theory and $q$-deformed Onsager algebras, respectively. 
Specific examples of particle models leading to rational Bethe Ansatz equations of type B are  the
open quantum nonlinear Schr\"odinger equation  \cite{gau:boundary,gau:bethe,die-ems:orthogonality}
and its lattice discretization
\cite{doi-fio-rav:generalized}. The trigonometric/hyperbolic type B systems 
arise in turn as Bethe Ansatz equations for the 
 $q$-boson model with open-end boundary interactions
\cite{li-wan:exact,die-ems:orthogonality,die-ems-zur:completeness}. The open Heisenberg  XXX and XXZ spin chains with boundary interactions
\cite{alc-bar-bat-bax-qui:surface,skl:boundary,cao-lin-shi-wan:exact,nep:bethe,mel-tib-mar:bethe,fra-see-wir:separation,bel-cra-rag:algebraic}
again lead to  rational and trigonometric/hyperbolic Bethe Ansatz equations of type B that do not belong to the convex regime considered here, and thus fall outside the scope of the results presented below.

\subsection{Bethe equations}\label{BE-B}
The  Bethe system of type B is defined as 
\begin{equation}\label{BAE-B}
\boxed{e^{2i\alpha  \xi_j }
=
(-1)^{\varepsilon}
\prod_{1\le k \le K}
 \frac{\text{s}  (ia_k+\xi_j)}{\text{s}(ia_k-\xi_j)}
\prod_{\substack{1\le j^\prime \le n, \, j^\prime\neq j \\ 1 \le l \le L}}
 \frac{\text{s}(ib_{l}  + \xi_j+\xi_{j^\prime})}{\text{s}(ib_{l}- \xi_j - \xi_{j^\prime} )} 
 \frac{\text{s}(ib_{l}  + \xi_j-\xi_{j^\prime})}{\text{s}(ib_{l}- \xi_j + \xi_{j^\prime} )} }
\end{equation}
($j=1,\dots, n$). Here $\text{s}(x)=\text{s}^{(u)}(x)$ is taken from Eq. \eqref{S}, and in order to ensure
%the irreducibility and 
the convexity of the system
we will again
%assume that  $L>0$ and
impose the following restrictions on the $L+K+2$ parameters $\alpha$, $\varepsilon$, $a_1,\ldots ,a_{K}$ and
$b_1,\ldots ,b_L$:
\begin{equation}\label{convex-B}
\boxed{\alpha\in (0,\infty ) ,\quad \varepsilon\in \{ 0,1\} , \quad  a_k, b_l  \in \begin{cases}  (0,\infty) &\text{if}\ u=r \\ (0,\pi )&\text{if}\ u=h \\  (0,\infty )\cup\big( (0,\infty)+i\pi \bigr) &\text{if}\ u=t  \end{cases}}
\end{equation}
($k=1,\ldots ,K$, $l=1,\ldots, L$).

\subsection{Bethe roots}\label{BR-B}
Following the same pattern as for type A, we establish the Morse function whose equation for the critical points reproduces the Bethe system of type B (after exponentiation):
\begin{align}\label{Morse-B}
V^{(u)}_{\text{B},\mu} (\boldsymbol{\xi};\alpha,\varepsilon):=&
\sum_{1\leq j\leq n}  \Biggl(
\alpha \xi_j^2-2\pi \left(\mu_j+\frac{\varepsilon}{2} \right) \xi_j 
+  \sum_{1\leq k\leq K} \int^{\xi_j}_0 v_{a_k}(x)\, \text{d}x\Biggr) \\
&+  \sum_{\substack{1\le j < j^\prime \le n \\ 1\leq l\leq L}} \Biggl( \int_0^{\xi_j+\xi_{j^\prime}} v_{b_l}(x)\, \text{d}x +
     \int_0^{\xi_j-\xi_{j^\prime}} v_{b_l}(x)\, \text{d}x \Biggr)
  ,\nonumber
\end{align}
with $v_a(x)=v_a^{(u)}(x)$ as in Subsection \ref{BR-A} and
 \begin{equation}\label{weights-B}
\mu\in \Lambda_{\text{B}}:=   
\{ \mu\in\mathbb{Z}^n \mid   \mu_1>\mu_2> \cdots > \mu_n>  0\} .
\end{equation}
As before, the parameter restrictions in Eq. \eqref{convex-B} guarantee that (i)
$V^{(u)}_{\text{B},\mu} (\boldsymbol{\xi};\alpha)$   is smooth in $\boldsymbol{\xi}\in \mathbb{R}^n$, (ii)
$V^{(u)}_{\text{B},\mu} (\boldsymbol{\xi};\alpha)\to \infty$ when $|\boldsymbol{\xi}|\to\infty$, and (iii) 
$V^{(u)}_{\text{B},\mu} (\boldsymbol{\xi};\alpha)$ is strictly convex:
\begin{align}\label{Hesse-B}
&H^{\text{B}}_{j,j^\prime} (\boldsymbol{\xi}):=\partial_{\xi_j}\partial_{\xi_{j^\prime}} V^{(u)}_{\text{B},\mu} (\boldsymbol{\xi};\alpha ,\varepsilon ) \\
&={\footnotesize
\begin{cases}
2\alpha + \sum_{1\le k \le K}v_{a_k}'(\xi_j)  
+
 \sum_{\substack{\ell\neq j\\ 1\le l \le L}} \bigl( v_{b_l}'(\xi_j+\xi_\ell)+v_{b_l}'(\xi_j-\xi_\ell) \bigr) & \text{if $ j^\prime=j$,}\\
\sum_{1\le l \le L} \bigl( v_{b_l}'(\xi_{j^\prime}+\xi_j) + v_{b_l}'(\xi_{j^\prime}-\xi_j) \bigr) & \text{if $j^\prime\neq j$,}\\
\end{cases} }\nonumber
\end{align}
so
\begin{align*}
\sum_{1\leq j,j^\prime\leq n} & x_j x_{j^\prime}H^{\text{B}}_{j,j^\prime}    (\boldsymbol{\xi})
= \sum_{1\leq j\leq n} \Bigl(2 \alpha
+   \sum_{1\le k \le K}v_{a_k}'(\xi_j)  
 \Bigr) x_j^2  \\
&+  \sum_{\substack{1\leq j<j^\prime \leq n \\ 1\le l \le L}} 
\Bigl( v_{b_l}'(\xi_j +\xi_{j^\prime})(x_j+x_{j^\prime})^2+
v_{b_l}'(\xi_j -\xi_{j^\prime})(x_j-x_{j^\prime})^2 \Bigr) .
\end{align*}

Hence, for any $\mu\in\Lambda_{\text{B}}$ \eqref{weights-B} the critical
equation $\nabla_{\boldsymbol{\xi}} V^{(u)}_{\text{B},\mu} (\boldsymbol{\xi};\alpha, \varepsilon)=0$, i.e. the system
\begin{align}\label{CE-B}
2\alpha\xi_j +
 \sum_{1 \le k \le K} 
 v_{a_k}(\xi_j) 
+ \sum_{\substack{1\le j^\prime \le n,\,  j^\prime \neq j \\ 1 \le l\le L}}   \bigl( v_{b_l}(\xi_{j^\prime}+\xi_j)-v_{b_l}(\xi_{j^\prime}-\xi_j) \bigr)
=2\pi \left( \mu_j+\frac{\varepsilon}{2}\right)
\end{align}
($j=1,\dots, n$), has a unique solution $\boldsymbol{\xi}=\boldsymbol{\xi}^{(u)}_{\text{B},\mu}$ given by the global minimum
of the strictly convex Morse function $V^{(u)}_{\text{B},\mu} (\boldsymbol{\xi};\alpha, \varepsilon)$ \eqref{Morse-B}.

\begin{proposition}[Bethe roots of type B]\label{Bethe-roots-B:prp}
(i) For the parameter regime in Eq. \eqref{convex-B} and
any $\mu\in\Lambda_{\text{B}}$ \eqref{weights-B}, the unique global minimum $\boldsymbol{\xi}=\boldsymbol{\xi}^{(u)}_{\text{B},\mu}$ of $V^{(u)}_{\text{B},\mu} (\boldsymbol{\xi};\alpha ,\varepsilon)$ \eqref{Morse-B} produces a solution of the
 Bethe system of type B \eqref{BAE-B}.

(ii) The assignment $\mu\to \boldsymbol{\xi}^{(u)}_{\text{B},\mu}$, $\mu\in\Lambda_{\text{B}}$ is injective and  depends smoothly on the parameters
 \eqref{convex-B}.
\end{proposition}

\begin{proof}
The proof of Proposition \ref{Bethe-roots-A:prp} applies verbatim upon substituting: A $\to$ B and Eqs. \eqref{BAE-A}, \eqref{convex-A}, \eqref{Morse-A}, \eqref{Hesse-A}, \eqref{CE-A}
$\to$ Eqs. \eqref{BAE-B}, \eqref{convex-B}, \eqref{Morse-B}, \eqref{Hesse-B}, \eqref{CE-B}, respectively.
\end{proof}

\subsection{Bethe bounds}\label{BB-B}
In the same way as in Section \ref{sec2}, we deduce from
the system in Eq. \eqref{CE-B} that the global minimum
$\boldsymbol{\xi}^{(u)}_{\text{B},\mu}$ of $V^{(u)}_{\text{B},\mu} (\boldsymbol{\xi};\alpha ,\varepsilon)$ \eqref{Morse-B} with  $\mu\in\Lambda_{\text{B}}$ \eqref{weights-B}
belongs to the open cone
\begin{equation}\label{wedge-B}
\mathbb{A}_{\text{B}}:= \{ \boldsymbol{\xi}\in\mathbb{R}^n\mid \xi_1 >\xi_2>\cdots >\xi_n>0 \} .
\end{equation}
Indeed, it is manifest from Eq. \eqref{CE-B}  that $\xi_j>0$ if $\mu_j>0$, because the expression on the LHS is monotonously increasing and odd in $\xi_j$ (for our parameter regime).
Moreover,  subtracting the $j^\prime$th equation from the $j$th equation now yields that
\begin{align}\label{CED-B}
2\alpha (\xi_j -\xi_{j^\prime}) &+
 \sum_{1 \le k \le K} 
\bigl(  v_{a_k}(\xi_j)- v_{a_k}(\xi_{j^\prime}) \bigr)  \\
&+\sum_{\substack{1\le j^{\prime\prime} \le n,\ j^{\prime\prime}\neq j,j^\prime \\ 1 \le l\le L}}  \bigl(  v_{b_l}(\xi_j+\xi_{j^{\prime\prime}})-v_{b_l}(\xi_{j^\prime}+\xi_{j^{\prime\prime}})\bigr) 
\nonumber\\
&+\sum_{\substack{1\le j^{\prime\prime} \le n\\ 1 \le l\le L}}  \bigl(  v_{b_l}(\xi_j-\xi_{j^{\prime\prime}})-v_{b_l}(\xi_{j^\prime}-\xi_{j^{\prime\prime}})\bigr)
=2\pi (\mu_j-\mu_{j^\prime}) ,\nonumber
\end{align}
so by monotonicity
$\xi_j > \xi_{j^\prime}$ if $\mu_j > \mu_{j^\prime}$  as before. Refining the analysis finally leads us to the following bounds on $\xi_j$ and on the gaps between $\xi_j$ and $\xi_{j^\prime}$ at 
$\boldsymbol{\xi}=\boldsymbol{\xi}^{(u)}_{\text{B},\mu}$.

\begin{proposition}[Bethe bounds of type B]\label{Bethe-bounds-B:prp}
For the parameter regime in Eq. \eqref{convex-B} and
any $\mu\in\Lambda_{\text{B}}$ \eqref{weights-B}, one has that at the global minimum $\boldsymbol{\xi}=\boldsymbol{\xi}^{(u)}_{\text{B},\mu}$ of
$V^{(u)}_{\text{B},\mu} (\boldsymbol{\xi};\alpha ,\varepsilon )$ \eqref{Morse-B}:
\begin{subequations}
\begin{equation}\label{bounds-B:a}
\frac{\pi\left( \mu_j+\frac{\varepsilon}{2}\right) }{\alpha+\kappa^{(u)}_-} \le \xi_j \le \frac{\pi\left( \mu_j+\frac{\varepsilon}{2}\right)}{\alpha+\kappa^{(u)}_+}
\end{equation}
for $1\leq j\leq n$, and
\begin{equation}\label{bounds-B:b}
\frac{\pi(\mu_j-\mu_{j^\prime})}{\alpha+\kappa^{(u)}_-} \le \xi_j-\xi_{j^\prime} \le \frac{\pi(\mu_j-\mu_{j^\prime})}{\alpha+\kappa^{(u)}_+}
\end{equation}
\end{subequations}
for $1\leq j<j^\prime\leq n$, where
\begin{equation*}
\kappa^{(u)}_-  := \begin{cases}
\sum_{1\leq k\leq K} a_k^{-1}+2(n-1) \sum_{1\leq l\leq L}   b_l^{-1}&\text{if}\ u=r, \\
\frac{1}{2}\sum_{1\leq k\leq K} \cot (\frac{1}{2} a_k)+(n-1) \sum_{1\leq l\leq L}   \cot (\frac{1}{2}b_l) &\text{if}\ u=h, \\
\frac{1}{2}\sum_{1\leq k\leq K} \coth (\frac{1}{2} \emph{Re} (a_k))+(n-1) \sum_{1\leq l\leq L}   \coth (\frac{1}{2}\emph{Re} (b_l)) &\text{if}\ u=t ,\end{cases}
\end{equation*}
and
\begin{equation*}
\kappa^{(u)}_+ := \begin{cases}
0&\text{if}\ u=r\ \text{or}\ u=h, \\
\frac{1}{2}\sum_{1\leq k\leq K} \tanh (\frac{1}{2} \emph{Re}(a_k))   & \\
\quad + (n-1) \sum_{1\leq l\leq L}   \tanh (\frac{1}{2}\emph{Re}(b_l)) &\text{if}\ u=t . \end{cases}
\end{equation*}
\end{proposition}
\begin{proof} As in the proof of Proposition \ref{Bethe-bounds-B:prp}, one deduces---from the bounds on the derivative of $v_a^{(u)
}$ \eqref{v} in combination the mean value theorem---via Eq. \eqref{CE-B} that
\begin{equation*}
(\alpha+\kappa^{(u)}_+) \xi_j \leq \pi\left( \mu_j+\frac{\varepsilon}{2}\right) \leq  (\alpha+\kappa^{(u)}_-) \xi_j
\end{equation*}
for $1\leq j\leq n$,
and via Eq. \eqref{CED-B} that
\begin{equation*}
(\alpha+\kappa^{(u)}_+) ( \xi_j -\xi_{j^\prime}) \leq \pi ( \mu_j-\mu_{j^\prime}) \leq  (\alpha+\kappa^{(u)}_-) (\xi_j-\xi_{j^\prime}) 
\end{equation*}
for $1\leq j<j^\prime\leq n$.
\end{proof}

\begin{remark}\label{complex-parameters:rem}
For $a$ complex and $x$ real, one has that
\begin{equation*}
\text{Re} \bigl( v_a^\prime (x) \bigr)= \frac{1}{2} \Bigl(      v_{\text{Re}(a)}^\prime \bigl(x+\text{Im}(a)\bigr)   +  v_{\text{Re}(a)}^\prime \bigl(x-\text{Im}(a)\bigr)       \Bigr)  .
\end{equation*}
The upshot is that the propositions in Sections \ref{sec2} and \ref{sec3} persist for complex parameters $a_k$ and $b_l$, provided both all non-real parameters $a_k$ and all non-real parameters  $b_l$ arise in complex conjugate pairs.  Indeed, we may in this situation relax the parameter restrictions in Eqs. \eqref{convex-A} and \eqref{convex-B}
by replacing
 $a_k$ with  $\text{Re} (a_k)$ and $b_l$  with $\text{Re} (b_l)$, while performing the same modifications in the expressions for the bounds in Propositions \ref{Bethe-bounds-A:prp} and \ref{Bethe-bounds-B:prp} (when $u=r$ and $u=h$).
\end{remark}

\section{Estimates for the zeros of Askey-Wilson and Wilson polynomials}\label{sec4}
The estimates for the zeros of the (Askey-)Wilson polynomials hinge on algebraic Bethe equations that arise
from the type B system in Eqs.  \eqref{BAE-B}, \eqref{convex-B} via the degeneration
$\alpha \to 0$. For this purpose, it is enough to restrict attention to the case that $\varepsilon=0$, which will therefore be assumed from now on (unless explicitly stated otherwise).

\subsection{Algebraic Bethe system of type B at $\alpha =0$}
The following proposition adapts the results of Section \ref{sec3} for $u=r$ and $u=t$ to the case $\alpha = 0$ with
$\mu\in \Lambda_B$ \eqref{weights-B} minimal:
\begin{equation}\label{rho-B}
\mu=\rho:=(n,n-1,n-2,\ldots ,2,1) 
\end{equation}
(and  $\varepsilon=0$).

\begin{proposition}[Bethe system of type B at $\alpha =0$]\label{alpha=0-Bethe-roots-B:prp}
Let  $K>2$, $L> 0$, and let $a_1,\ldots ,a_{K}$ and $b_1,\ldots ,b_L$ satisfy the restrictions in Eq. \eqref{convex-B}.
For $u=r$ and $u=t$ the unique global minimum $\boldsymbol{\xi}_{\text{B},\rho}^{(u)}$ of the strictly convex Morse function
$V^{(u)}_{\text{B},\rho} (\boldsymbol{\xi}; 0,0)$ \eqref{Morse-B}, \eqref{rho-B}
produces a solution to the algebraic Bethe system \eqref{BAE-B} at $\alpha =0$ and $\varepsilon=0$, which depends smoothly on the parameters 
$a_1,\ldots ,a_{K}$ and $b_1,\ldots ,b_L$.
Moreover, at $\boldsymbol{\xi}=\boldsymbol{\xi}_{\text{B},\rho}^{(u)}$  the following inequalities are satisfied:
\begin{subequations}
\begin{equation}
\frac{\pi (n+1-j)}{\kappa^{(r)}_-} \le \xi_j 
\end{equation}
($1\leq j\leq n$) and
\begin{equation}
\frac{\pi(j^\prime-j)}{\kappa^{(r)}_-} \le \xi_j-\xi_{j^\prime} 
\end{equation}
\end{subequations}
($1\leq j<j^\prime\leq n$) if $u=r$, and
\begin{subequations}
\begin{equation}
\frac{\pi (n+1-j)}{\kappa^{(t)}_-} \le \xi_j \le \frac{\pi (n+1-j)}{\kappa^{(t)}_+}
\end{equation}
($1\leq j\leq n$) and
\begin{equation}
\frac{\pi(j^\prime-j)}{\kappa^{(t)}_-} \le \xi_j-\xi_{j^\prime} \le \frac{\pi(j^\prime-j)}{\kappa^{(t)}_+}
\end{equation}
\end{subequations}
($1\leq j<j^\prime\leq n$) if $u=t$, with $\kappa^{(u)}_\pm$ taken from Proposition \ref{Bethe-bounds-B:prp}.
\end{proposition}

\begin{proof}
We only need to check the existence of the unique global minimum of the Morse function at $(\alpha ,\varepsilon) =(0,0)$, since the rest of the proof
can be extracted verbatim from the proofs of
Propositions \ref{Bethe-roots-B:prp} and  \ref{Bethe-bounds-B:prp} via the specialization $\mu=\rho$,  $\alpha=0$, and  $\varepsilon =0$. It is moreover clear from the Hessian \eqref{Hesse-B} that the Morse function $V^{(u)}_{\text{B},\mu} (\boldsymbol{\xi}; 0,0)$ \eqref{Morse-B} remains strictly convex if $K>0$, which automatically settles the question of the minimum's uniqueness.
For the existence  to persist, it is enough to infer that $V^{(u)}_{\text{B},\rho} (\boldsymbol{\xi}; 0,0)\to \infty$  for $|\boldsymbol{\xi}|\to\infty$.  

When $u=t$ the quadratic growth of the integral $\int_0^\xi v^{(t)}_a(x) \text{d}x$ as $|\xi |\to \infty$---which stems from
the quasi-periodicity $v^{(t)}_a(x+2\pi)=v^{(t)}_a(x)+2\pi$---immediately guarantees the desired growth of $V^{(t)}_{\text{B},\mu} (\boldsymbol{\xi}; 0,0)$  for $K> 0$, therewith confirming  the existence of the (unique) global minimum of  our Morse function for any  $\mu\in\Lambda_B$ in this situation.

When $u=r$ one has  that  $\lim_{x\to \infty} v^{(r)}_a(x)=\pi $. This case is therefore more subtle, as the existence of the global minimum of $V^{(r)}_{\text{B},\mu} (\boldsymbol{\xi}; 0,0)$  is no longer guaranteed for all $\mu\in\Lambda_B$ \eqref{weights-B}.  We notice though that---apart from the linear term---our Morse function is symmetric with respect to the natural action of the hyperoctahedral group of signed permutations on the components of $\boldsymbol{\xi}$ (because $\int_0^\xi v^{(r)}_a(x) \text{d}x$ is even in $\xi$).
It thus suffices to infer that $V^{(r)}_{\text{B},\rho} (\boldsymbol{\xi};0,0) \to \infty$ when $|\boldsymbol{\xi}|\to\infty$ on the closure
of the fundamental cone $\mathbb{A}_{\text{B}}$ \eqref{wedge-B}. Notice in this connection also that  the linear term $\rho_1\xi_1+\cdots +\rho_n\xi_n$ 
grows fastest for $|\boldsymbol{\xi}|\to \infty$  on this fundamental cone, because
\begin{equation}\label{fund-dominance-B}
\rho_1(\xi_1-\epsilon_1\xi_{\sigma_1})+\cdots +\rho_n(\xi_n-\epsilon_n\xi_{\sigma_n})\geq  0
\end{equation}
 for all $\boldsymbol{\xi}\in\mathbb{A}_{\text{B}}$ \eqref{wedge-B}, $ \{ \sigma_1,\ldots ,\sigma_n\} =\{ 1,\ldots ,n\} $ and $\epsilon_j\in \{ 1,-1\}$ ($j=1,\ldots, n$).

To verify the unbounded growth of  $V^{(r)}_{\text{B},\rho} (\boldsymbol{\xi};0,0)$ on the closure of  $\mathbb{A}_{\text{B}}$, we set
$x_j=\xi_j-\xi_{j+1}$ ($j=1,\ldots ,n$) with the convention that $\xi_{n+1}\equiv 0$. Moreover, for a given nonempty subset
$J\subset \{ 1,\ldots ,n\}$,  let us write  $\boldsymbol{\xi}\stackrel{J}{\to} \infty$ if $x_j\to\infty$ for $j\in J$ while $x_j$ remains bounded for $j\not\in J$. 
Since
\begin{equation*}
\xi_j=x_j+\cdots +x_n 
\end{equation*}
$(1\leq j\leq n)$ and
\begin{align*}
 \xi_j-\xi_{j\prime}&=x_j+\cdots +x_{j^\prime-1} , \\
 \xi_j+\xi_{j\prime}&= x_j+\cdots +x_{j^\prime-1} +2 \bigl(   x_{j^\prime}+\cdots +x_n    \bigr)
\end{align*}
$(1\leq j<j^\prime \leq n) $,
it follows that for $\boldsymbol{\xi}\stackrel{J}{\to} \infty$
\begin{align*}
\sum_{1\leq j\leq n} \int_0^{\xi_j} v^{(r)}_{a_k}(x)\text{d}x &\sim \pi \sum_{j\in J} jx_j ,  \\
\sum_{1\leq j<j^\prime\leq n}   \int_0^{\xi_j- \xi_{j^\prime}}  v^{(r)}_{b_l} (x)\text{d}x &\sim \pi \sum_{j\in J} j(n-j) x_j ,\\
\sum_{1\leq j<j^\prime\leq n}   \int_0^{\xi_j +\xi_{j^\prime}}  v^{(r)}_{b_l} (x)\text{d}x & \sim \pi \sum_{j\in J} j(n-1) x_j .
\end{align*}
Hence, the leading asymptotics
of $V^{(r)}_{\text{B},\rho} (\boldsymbol{\xi};0,0)$ 
for $\boldsymbol{\xi}\stackrel{J}{\to} \infty$ is given by
\begin{align*}
& \pi \sum_{j\in J} jx_j  \bigl(    -(2n+1-j)  +K+ L(2n-1-j)                 \bigr)\geq  \pi \sum_{j\in J} jx_j
\end{align*}
(where for the last estimate it was used that $K> 2$ and $L> 0$). By varying our choice for $J$, this diverging lower bound on the leading asymptotics confirms that 
$V^{(r)}_{\text{B},\rho} (\boldsymbol{\xi};0,0) \to \infty$ for $|\boldsymbol{\xi}|\to\infty$ (first on the closure of the fundamental cone $\mathbb{A}_{\text{B}}$, and then on the whole $\mathbb{R}^n$
because of the hyperoctahedral symmetry of the nonlinear terms and Eq. \eqref{fund-dominance-B}).
\end{proof}

\begin{remark}
In the hyperbolic case one has that $\lim_{x\to \infty} v^{(h)}_a(x)=\pi -a$ (with $0<a<\pi$), which entails that
the existence of
the (unique) global minimum of $V^{(h)}_{\text{B},\rho} (\boldsymbol{\xi}; 0,0)$  is only guaranteed for $K>2$ and $L>0$  provided
 the parameters $a_k$ and $b_l$ lie  sufficiently close to $0$ in the interval $(0,\pi)$.
\end{remark} 

\begin{remark}\label{regularity-B} It is clear (from the stated inequalities) that the global minimum $\boldsymbol{\xi}=\boldsymbol{\xi}_{\text{B},\rho}^{(u)}$
of  $V^{(u)}_{\text{B},\rho} (\boldsymbol{\xi};0,0)$ in Proposition \ref{alpha=0-Bethe-roots-B:prp}
 is again assumed inside the cone
$\mathbb{A}_{\text{B}}$ \eqref{wedge-B}, both for $u=r$ and $u=t$. In the latter (trigonometric) case, the Bethe roots under consideration belong in fact to the interval $(0,\pi)$, i.e. one has that
\begin{equation}
\pi >\xi_1>\xi_2>\cdots > \xi_n>0
\end{equation}
at the critical point of $V^{(t)}_{\text{B},\rho} (\boldsymbol{\xi};0,0)$. Indeed, for $u=t$ the LHS of Eq. \eqref{CE-B}  takes the value $\bigl(2\alpha +K +2L(n-1)\bigr)\pi$ at $\xi_j=\pi$ (by the quasi-periodicity $v^{(t)}_a(x+2\pi)=v^{(t)}_a(x)+2\pi$ of the odd function $v^{(t)}_a(x)$ \eqref{v}), while the RHS takes only
the value $2\pi \left(n+1-j+\frac{\epsilon}{2}\right)$ when $\mu=\rho$ \eqref{rho-B}. Hence, it is immediate from the monotonicity of the LHS as function of $\xi_j$ that  the corresponding Bethe solutions must be smaller than $\pi$ when
$K>2 $ and $L>0$.
\end{remark}

\subsection{Wilson polynomials}
The Wilson polynomial \cite{wil:some}, \cite[\text{Ch. 9.1}]{koe-swa:hypergeometric}
\begin{align}\label{Wp}
\text{p}_n(\xi ;a,b,c,d):=&\frac{(-1)^n (a+b,a+c,a+d)_n}{(n+a+b+c+d-1)_n} \\
 &\times  {}_4F_3
\left[ \begin{array}{c} -n,n+a+b+c+d-1,a+i\xi ,a-i\xi \\
a+b,a+c,a+d \end{array}  ; 1 \right] \nonumber
\end{align}
is a monic polynomial of degree $n$ in $\xi^2$ that satisfies the second-order difference equation
\begin{equation}\label{DE-W}
A(\xi) \bigl( \text{p}_n(\xi+i) -\text{p}_n(\xi)\bigr)+
A(-\xi) \bigl( \text{p}_n(\xi-i) -\text{p}_n(\xi)\bigr) = E_n \text{p}_n(\xi)
\end{equation}
with
\begin{align*}
A(\xi )&= \frac{(\xi+ia)(\xi+ib)(\xi+ic)(\xi+id)}{2\xi(2\xi+i)} , \\
E_n&=-n(n+a+b+c+d-1) 
\end{align*}
(as a rational identity in the parameters $a,b,c,d$).
For $a,b,c,d>0$ the Wilson polynomials $\text{p}_n(\xi;a,b,c,d)$, $n=0,1,2,\ldots$ are (manifestly) analytic in the parameters and  constitute an orthogonal basis on the interval $(0,\infty)$ with respect to the weight function
\begin{equation}\label{weight-W}
\Delta (\xi)=\left| \frac{\Gamma (a+i\xi)\Gamma (b+i\xi)\Gamma (c+i\xi)\Gamma (d+i\xi)} {\Gamma (2i\xi )}\right|^2
\end{equation}
(where $\Gamma (\cdot)$ refers to the gamma function).

\begin{remark}\label{spectral-problem:rem}
It is helpful to view
the difference equation \eqref{DE-W}  as an eigenvalue equation  with eigenvalue $E_n$ for Wilson's second-order difference operator $D$ acting on $\text{p}_n(\xi)$ at the LHS.
The fact that the monic Wilson polynomials $\text{p}_n(\xi;a,b,c,d)$, $n=0,1,2,\ldots$ solve the eigenvalue equations in question implies that $D$
preserves the space of even polynomials in $\xi$ without raising the degree.
Since the corresponding eigenvalues $E_n$ are nondegenerate for $a,b,c,d>0$,
this triangularity of  the Wilson operator with respect to the monomial basis $\xi^{2n}$, $n=0,1,2,\ldots$ guarantees that its eigenbasis 
is unique in the space of even polynomials (up to normalization):
\begin{equation}
\text{p}_n(\xi ;a,b,c,d)  =  \left(  \prod_{0\leq m  <n}   \frac{D-E_m}{E_n-E_m} \right) \xi^{2n} ,\quad n=0,1,2,\ldots
\end{equation}
(using, e.g., the Cayley-Hamilton theorem in the finite-dimensional invariant subspace of polynomials of degree at most $n$ in $\xi^2$).
\end{remark}

\begin{theorem}[Zeros of the Wilson polynomials]\label{W:thm}
For $a,b,c,d>0$, the zeros
$
\xi_1^{(n)}>\xi_2^{(n)}>\cdots >\xi_n^{(n)}>0
$
of the Wilson  polynomial $\emph{p}_n(\xi;a,b,c,d)$ \eqref{Wp} obey the following inequalities:
\begin{subequations}
\begin{equation}
\frac{\pi (n+1-j)}{k^{(n)}_-(a,b,c,d)} \le \xi_j^{(n)} 
\end{equation}
($1\leq j\leq n$) and
\begin{equation}
\frac{\pi(j^\prime-j)}{k^{(n)}_-(a,b,c,d)} \le \xi_j^{(n)}-\xi_{j^\prime}^{(n)}
\end{equation}
($1\leq j<j^\prime\leq n$),
where
\begin{equation}
k^{(n)}_-(a,b,c,d):=2(n-1)+a^{-1}+b^{-1}+c^{-1}+d^{-1}.
\end{equation}
\end{subequations}
\end{theorem}
\begin{proof}
Let  $\boldsymbol{\xi}=\boldsymbol{\xi}_{\text{B},\rho}^{(r)}$ denote the global minimum of $V^{(r)}_{\text{B},\rho} (\boldsymbol{\xi}; 0,0)$ \eqref{Morse-B}, \eqref{rho-B} for
$K=4$ and $L=1$, with
\begin{equation}\label{wilson-parameters}
a_1=a,\  a_2=b,\  a_3=c,\  a_4=d\quad\text{and}\quad b_1=1.
\end{equation}
Following \cite{oda-sas:equilibria,die:equilibrium,sas-yan-zha:bethe,bih-cal:properties:a}, we shall check that  the associated polynomial
\begin{equation}\label{pn}
\text{p}_n(\xi)=(\xi^2-\xi_1^2)(\xi^2-\xi_2^2)\cdots (\xi^2-\xi_n^2)
\end{equation}
satisfies the difference equation  \eqref{DE-W}. 
To this end we first observe that the substitution of $\text{p}_n(\xi)$ into Eq. \eqref{DE-W} gives rise to a polynomial relation of degree $n$ in $\xi^2$ (cf. Remark \ref{spectral-problem:rem}).
Since $\xi_1>\xi_2>\cdots >\xi_n>0$ (cf. Remark \ref{regularity-B}), it is sufficient to infer the polynomial relation in question at $\xi=\xi_j$ ($j=1,\ldots ,n$). This entails the following algebraic relations
\begin{equation}\label{BAE-W}
A  (\xi_j) \text{p}_n(\xi_j+i) +A(-\xi_j)\text{p}_n(\xi_j-i)=0\quad (j=1,\ldots, n),
\end{equation}
which in turn hold by Proposition \ref{alpha=0-Bethe-roots-B:prp}. Indeed---upon making $A(\cdot)$ and $\text{p}_n(\cdot )$ explicit---it is readily seen that Eq. \eqref{BAE-W} amounts
precisely to the $K=4$, $L=1$ Bethe system of type B with parameters \eqref{wilson-parameters} that is solved by $\boldsymbol{\xi}=\boldsymbol{\xi}_{\text{B},\rho}^{(r)}$.

The upshot is that $\text{p}_n(\xi;a,b,c,d)$ \eqref{Wp} and  $\text{p}_n(\xi)$ \eqref{pn} are both monic polynomials of degree $n$ in $\xi^2$ satisfying the eigenvalue equation \eqref{DE-W}. Since the corresponding eigenvalues $E_n$, $n=0,1,2,\ldots$ are nondegenerate for $a,b,c,d>0$, this implies (cf. Remark \ref{spectral-problem:rem}) that  $\text{p}_n(\xi;a,b,c,d)=\text{p}_n(\xi)$, i.e. $\xi_j^{(n)}=\xi_j$ ($j=1,\ldots ,n$). The asserted inequalities for the zeros are now inherited from those  for  $\boldsymbol{\xi}=\boldsymbol{\xi}_{\text{B},\rho}^{(r)}$ in
Proposition \ref{alpha=0-Bethe-roots-B:prp}.
\end{proof}

\subsection{Askey-Wilson polynomials}
The Askey-Wilson polynomial \cite{ask-wil:some}, \cite[\text{Ch. 14.1}]{koe-swa:hypergeometric}
\begin{align}\label{AWp}
\text{p}_n(\xi ;a,b,c,d;q):=&\frac{(ab,ac,ad ;q)_n}{(2a)^n (abcdq^{n
-1};q)_n}     \\ 
&\times {}_4\Phi_3
\left[ \begin{array}{c} q^{-n},abcdq^{n-1},ae^{i\xi},ae^{-i\xi} \\
ab,ac,ad \end{array}  ; q,q  \right]  \nonumber
\end{align}
is a monic polynomial of degree $n$ in $\cos(\xi)$. It  satisfies the second-order difference equation
\begin{equation}\label{DE-AW}
A(\xi) \bigl( \text{p}_n(\xi-i\log (q)) -\text{p}_n(\xi)\bigr)+
A(-\xi) \bigl( \text{p}_n(\xi+i\log (q)) -\text{p}_n(\xi)\bigr) = E_n \text{p}_n(\xi)
\end{equation}
with
\begin{align*}
A(\xi )&= \frac{(1-a e^{i\xi})(1-b e^{i\xi})(1-c e^{i\xi})(1-d e^{i\xi})}{(1-e^{2i\xi})(1-q e^{2i\xi})} , \\
E_n&=q^{-n}(1-q^n)(1-abcdq^{n-1})
\end{align*}
(as a rational identity in the parameters $a,b,c,d$ and $q$). For $-1<a,b,c,d,q<1$, the Askey-Wilson polynomials
are analytic in the parameters and
constitute an orthogonal basis  on the interval $(0,\pi)$ with respect to the weight function
\begin{equation}
\Delta (\xi)=\left|  \frac{(e^{2i\xi };q)_\infty}{(ae^{i\xi },b e^{i\xi},ce^{i\xi},de^{i\xi};q)_\infty}   \right|^2 .
\end{equation}
Notice that this means in particular that the singularities of $\text{p}_n(\xi ;a,b,c,d;q)$ \eqref{AWp} at $a=0$ and $q=0$ are in fact removable.

\begin{theorem}[Zeros of the Askey-Wilson polynomials]\label{AW:thm}
For $-1<a,b,c,d,q<1$, the zeros
$
\pi >\xi_1^{(n)}>\xi_2^{(n)}>\cdots >\xi_n^{(n)}>0
$
of the Askey-Wilson  polynomial $\emph{p}_n(\xi;a,b,c,d;q)$ \eqref{AWp} obey the following inequalities:
\begin{subequations}
\begin{equation}
 \frac{\pi (n+1-j) }{k_-^{(n)}(a,b,c,d;q)}  \leq \xi^{(n)}_j \leq  \frac{\pi (n+1-j) }{k_+^{(n)}(a,b,c,d;q)} 
\end{equation}
($1\leq j\leq n$) and
\begin{equation}
\frac{\pi(j^\prime-j)}{k^{(n)}_-(a,b,c,d;q)} \le \xi^{(n)}_j-\xi^{(n)}_{j^\prime} \le \frac{\pi(j^\prime-j)}{k^{(n)}_+(a,b,c,d;q)}
\end{equation}
($1\leq j<j^\prime\leq n$),
where
\begin{align}\label{kpm}
k_{\pm}^{(n)}  & (a,b,c,d;q):= (n-1)  \Biggl(\frac{1-|q|}{1+ |q|}\Biggr)^{\pm 1}+ \\
&\frac{1}{2}\left( \Biggl(\frac{1-|a|}{1+ |a|}\Biggr)^{\pm 1}+ \Biggl(\frac{1-|b|}{1+ |b|}\Biggr)^{\pm 1}+  \Biggl(\frac{1-|c|}{1+ |c|}\Biggr)^{\pm 1}+ \Biggl(\frac{1-|d|}{1+ |d|}\Biggr)^{\pm 1} \right).
\nonumber
\end{align}\end{subequations}
\end{theorem}

\begin{proof}
Let  $\boldsymbol{\xi}=\boldsymbol{\xi}_{\text{B},\rho}^{(t)}$ denote the global minimum of $V^{(t)}_{\text{B},\rho} (\boldsymbol{\xi}; 0,0)$ \eqref{Morse-B}, \eqref{rho-B} for
$K=4$ and $L=1$, with
\begin{equation}\label{askey-wilson-parameters}
e^{-a_1}=a,\  e^{-a_2}=b,\  e^{-a_3}=c,\  e^{-a_4}=d\quad\text{and}\quad e^{-b_1}=q
\end{equation}
(where it is temporarily assumed that $abcdq\neq 0$).
With the aid of  \cite{die:equilibrium} (cf. also \cite{ism-lin-roa:bethe}, \cite[Ch. 16.5]{ism:classical}, and \cite{oda-sas:equilibrium,sas-yan-zha:bethe,bih-cal:properties:b}),
 the proof of Proposition \ref{W:thm}  is readily adapted to the Askey-Wilson level starting from trigonometric polynomial
\begin{equation}\label{pn-AW}
\text{p}_n(\xi) =  \bigl( \cos (\xi) -\cos(\xi_1)\bigr)  \bigl( \cos (\xi) -\cos(\xi_2)\bigr)\cdots  \bigl( \cos (\xi) -\cos(\xi_n)\bigr) .
\end{equation}
Indeed, (similarly as before) the fact that $\text{p}_n(\xi)$ \eqref{pn-AW} solves
the second-order difference equation \eqref{DE-AW} encodes a trigonometric identity that
is equivalent to the following algebraic relations 
between the nodes $\pi >\xi_1>\xi_2>\cdots >\xi_1>0$:
\begin{equation}\label{BAE-AW}
A  (\xi_j) \text{p}_n\bigl(\xi_j-i\log (q)\bigr) +A(-\xi_j)\text{p}_n\bigl(\xi_j+i\log(q)\bigr)=0\quad (j=1,\ldots, n).
\end{equation}
Moreover, after making $A(\cdot)$ and $\text{p}_n(\cdot )$ explicit and invoking of the elementary trigonometric relation
$\cos (\xi )-\cos (\xi_j)=2\sin\bigl( \frac{1}{2}(\xi_j+\xi)\bigr) \sin\bigl( \frac{1}{2}(\xi_j-\xi)\bigr) $, the identities in
 Eq. \eqref{BAE-AW} become a manifest consequence of
 Proposition \ref{alpha=0-Bethe-roots-B:prp} upon identification with the
 trigonometric Bethe system of type B solved by $\boldsymbol{\xi}=\boldsymbol{\xi}_{\text{B},\rho}^{(t)}$.

Since for $a,b,c,d,q\in (-1,1)\setminus \{0\}$
the eigenvalues $E_n$  stemming from Eq. \eqref{DE-AW} are again nondegenerate, 
it is deduced through the difference equation that  $\text{p}_n(\xi;a,b,c,d;q)=\text{p}_n(\xi)$, and thus $\xi_j^{(n)}=\xi_j$ ($j=1,\ldots ,n$). (Here one uses again that the Askey-Wilson difference operator acting
on the LHS of Eq. \eqref{DE-AW} is triangular on the monomial basis $\cos^n (\xi)$, $n=0,1,2,\ldots$, cf. Remark \ref{spectral-problem:rem}.)
The asserted inequalities for the zeros now follow in the same manner as before from the formulas in
Proposition \ref{alpha=0-Bethe-roots-B:prp} through specialization. Finally, the resulting inequalities are readily extended to the case of one or more vanishing parameters $a,b,c,d,q$
by continuity.
\end{proof}

\section{Estimates for the zeros of the symmetric continuous Hahn polynomials}\label{sec5}
After performing the degeneration $\alpha\to 0$,
the rational Bethe system of type A in Eqs.  \eqref{BAE-A}, \eqref{convex-A} 
entails lower bounds for the zeros of the  symmetric continuous Hahn polynomials.

\subsection{Rational Bethe system of type A at $\alpha =0$}
The following proposition adapts the results of Section \ref{sec2} for $u=r$  to the case $\alpha = 0$  with
$\mu\in \Lambda_A$ \eqref{weights-A}  chosen minimal:
\begin{subequations}
\begin{equation}\label{rho-A}
\mu=\tilde{\rho}:=(\lfloor {\textstyle \frac{n-1}{2}}\rfloor ,\lfloor {\textstyle \frac{n-3}{2}}\rfloor ,\lfloor {\textstyle \frac{n-5}{2}}\rfloor ,\ldots ,\lfloor {\textstyle \frac{3-n}{2}}\rfloor ,\lfloor {\textstyle \frac{1-n}{2}}\rfloor ) 
\end{equation}
and
\begin{equation}\label{beta-n}
\beta =\beta_n :=\begin{cases}
\frac{1}{2}&\text{if}\ n=\text{even}, \\
0 &\text{if}\ n=\text{odd}.
\end{cases}
\end{equation}
\end{subequations}

\begin{proposition}[Rational Bethe system of type A at $\alpha =0$]\label{alpha=0-Bethe-roots-A:prp}
Let $K, L>0$ and let $a_1,\ldots ,a_{K}$ and $b_1,\ldots ,b_L$ satisfy the restrictions in Eq. \eqref{convex-A}.
The unique global minimum $\boldsymbol{\xi}_{\text{A},\tilde\rho}^{(r)}$ of the strictly convex Morse function
$V^{(r)}_{\text{A},\tilde\rho} (\boldsymbol{\xi}; 0,\beta_n )$ \eqref{Morse-A}, \eqref{rho-A}, \eqref{beta-n}
produces a solution to the algebraic Bethe system \eqref{BAE-A} at $\alpha =0$, which depends smoothly on the parameters 
$a_1,\ldots ,a_{K}$ and $b_1,\ldots ,b_L$.
Moreover, at $\boldsymbol{\xi}=\boldsymbol{\xi}_{\text{A},\tilde\rho}^{(r)}$  the following inequalities are satisfied:
\begin{equation}
\frac{2\pi(j^\prime-j)}{\kappa^{(r)}_-} \le \xi_j-\xi_{j^\prime} 
\end{equation}
($1\leq j<j^\prime\leq n$), with $\kappa^{(r)}_-$ taken from Proposition \ref{Bethe-bounds-A:prp}.
\end{proposition}

\begin{proof}
The statements are derived by adapting the proof of Proposition \ref{alpha=0-Bethe-roots-B:prp} to the present context.
Like before, it suffices to concentrate on the existence and the uniqueness of the global minimum of the Morse function at $\alpha =0$, as all other arguments
can be extracted verbatim from the proofs of
Propositions \ref{Bethe-roots-A:prp} and  \ref{Bethe-bounds-A:prp} via the specialization $\mu=\tilde\rho$, $\alpha=0$ and $\beta=\beta_n$. 
Since for $K\geq 1$ the strict convexity at $\alpha =0$ is manifest from
the Hessian \eqref{Hesse-A}, we will only infer the existence of the global minimum by verifying that
the Morse function $V^{(r)}_{\text{A},\tilde \rho} (\boldsymbol{\xi}; 0,\beta_n)\to\infty $   for $|\boldsymbol{\xi}|\to\infty$.   Upon exploiting the permutation symmetry of the nonlinear part of the Morse function, the required asymptotic analysis is restricted to the closure of the fundamental wedge $\mathbb{A}_{\text{A}}$ \eqref{wedge-A},
where we also use that the growth of the linear term $(\tilde{\rho}_1+\beta_n)\xi_1+\cdots + (\tilde{\rho}_n+\beta_n)\xi_n$ is maximal on this wedge because of
the inequality
\begin{equation}\label{fund-dominance-A}
(\tilde{\rho}_1+\beta_n)(\xi_1-\xi_{\sigma_1})+\cdots + (\tilde{\rho}_n+\beta_n)(\xi_n-\xi_{\sigma_n})\geq 0
\end{equation}
for all $\boldsymbol{\xi}\in\mathbb{A}_{\text{A}}$ and $\{ \sigma_1,\ldots ,\sigma_n\}=\{1,\ldots ,n\}$.

To verify the unbounded growth of  $V^{(r)}_{\text{A},\tilde\rho} (\boldsymbol{\xi};0,\beta_n)$ on the closure of  $\mathbb{A}_{\text{A}}$, we again set
$x_j=\xi_j-\xi_{j+1}$ ($j=1,\ldots ,n$) with the convention that $\xi_{n+1}\equiv 0$. Given a nonempty subset
$J\subset \{ 1,\ldots ,n\}$,  let us now write  $\boldsymbol{\xi}\stackrel{J}{\to} \infty$ if  $x_j\to\infty$ for $j\in J$ with $j<n$,  $|x_n|\to\infty$ for $n\in J$, 
while $x_j$ remains bounded for $j\not\in J$.
The
leading asymptotics
of $V^{(r)}_{\text{A},\tilde\rho} (\boldsymbol{\xi};0,\beta_n)$ 
for $\boldsymbol{\xi}\stackrel{J}{\to} \infty$ is then readily verified to dominate
\begin{equation*}
 \pi \sum_{j\in J} jx_j  \bigl(    -(n-j)  + L(n-j)                 \bigr)   
+  \begin{cases}  \pi K \sum_{j\in J} jx_j   &\text{if}\ n\not\in J ,\\
\pi K  |x_n| &\text{if}\ n\in J.
 \end{cases}  
 \end{equation*}
By varying our choice for $J$ (and assuming $K,L>0$), this diverging lower bound on the leading asymptotics confirms that 
$V^{(r)}_{\text{A},\rho} (\boldsymbol{\xi};0,\beta_n) \to \infty$ when $|\boldsymbol{\xi}|\to\infty$  (first on the closure of the fundamental wedge $\mathbb{A}_{\text{A}}$, and then on the whole $\mathbb{R}^n$
because of the permutation symmetry of the nonlinear terms and Eq. \eqref{fund-dominance-A}).
 \end{proof}

\subsection{Symmetric continuous Hahn polynomials}
The symmetric continuous Hahn polynomial \cite{ask-wil:set}, \cite[\text{Ch. 9.4}]{koe-swa:hypergeometric}
\begin{align}\label{cHp}
\text{p}_n(\xi ;a,b):=&\frac{i^n (2a,a+b)_n}{(n+2a+2b-1)_n} \\
 &\times  {}_3F_2
\left[ \begin{array}{c} -n,n+ 2a+2b-1,a+i\xi  \\
2a,a+b \end{array}  ; 1 \right] \nonumber
\end{align}
defines a monic polynomial of degree $n$ in $\xi$ satisfying the second-order difference equation
\begin{equation}\label{DE-cH}
A(\xi) \bigl( \text{p}_n(\xi+i) -\text{p}_n(\xi)\bigr)+
A(-\xi) \bigl( \text{p}_n(\xi-i) -\text{p}_n(\xi)\bigr) = E_n \text{p}_n(\xi)
\end{equation}
with
\begin{align*}
A(\xi )&= (\xi+ia)(\xi+ib) , \\
E_n&=-n(n+2a+2b-1) 
\end{align*}
(as a rational identity in the parameters $a,b$).
For $a,b>0$ the continuous Hahn polynomials $\text{p}_n(\xi;a,b)$, $n=0,1,2,\ldots$ are (manifestly) analytic in the parameters and form an orthogonal basis on $\mathbb{R}$ with respect to the weight function
\begin{equation}\label{weight-cH}
\Delta (\xi)=\left| \Gamma (a+i\xi)\Gamma (b+i\xi) \right|^2.
\end{equation}

\begin{theorem}[Zeros of the symmetric continuous Hahn polynomials]\label{cH:thm}
For $a,b>0$, the zeros
$
\xi_1^{(n)}>\xi_2^{(n)}>\cdots >\xi_n^{(n)}
$
of the symmetric continuous Hahn polynomial $\emph{p}_n(\xi;a,b)$ \eqref{cHp} obey the following inequalities:
\begin{subequations}
\begin{equation}
\frac{\pi(j^\prime-j)}{k^{(n)}_-(a,b)} \le \xi^{(n)}_j-\xi^{(n)}_{j^\prime} 
\end{equation}
($1\leq j<j^\prime\leq n$),
where
\begin{equation}
k^{(n)}_-(a,b):=n+a^{-1}+b^{-1} .
\end{equation}
\end{subequations}
\end{theorem}
\begin{proof} The proof is very similar as that for the Wilson polynomials in Theorem \ref{cH:thm}.
Let  $\boldsymbol{\xi}=\boldsymbol{\xi}_{\text{A},\tilde\rho}^{(r)}$ denote the global minimum of $V^{(r)}_{\text{A},\tilde\rho} (\boldsymbol{\xi}; 0,\beta_n)$ \eqref{Morse-A}, \eqref{rho-A}, \eqref{beta-n} for
$K=2$ and $L=1$, with
\begin{equation}\label{cH-parameters}
a_1=a,\  a_2=b \quad\text{and}\quad b_1=1.
\end{equation}
Following \cite{oda-sas:equilibria,sas-yan-zha:bethe}, we now infer that  the associated polynomial
\begin{equation}\label{pnA}
\text{p}_n(\xi)=(\xi-\xi_1)(\xi-\xi_2)\cdots (\xi-\xi_n)
\end{equation}
satisfies the difference equation  \eqref{DE-cH}. 
Indeed, (as before) substitution of $\text{p}_n(\xi)$ into Eq. \eqref{DE-cH} and inferring the corresponding polynomial relation of degree $n$ in $\xi$ at the nodes
 $\xi_1>\xi_2>\cdots >\xi_n$ gives rise to the algebraic relations
\begin{equation}\label{BAE-cH}
A  (\xi_j) \text{p}_n(\xi_j+i) +A(-\xi_j)\text{p}_n(\xi_j-i)=0\quad (j=1,\ldots, n).
\end{equation}
When making $A(\cdot)$ and $\text{p}_n(\cdot )$ explicit, Eq. \eqref{BAE-cH} (and thus Eq. \eqref{DE-cH}) is seen to hold
upon identification with the $K=2$, $L=1$ Bethe system of type A with parameters \eqref{cH-parameters} solved by $\boldsymbol{\xi}=\boldsymbol{\xi}_{\text{A},\tilde\rho}^{(r)}$.

In view of the triangularity of the action of the pertinent continuous Hahn difference operator on the monomial basis $\xi^n$, $n=0,1,2,\ldots$ (cf. Remark \ref{spectral-problem:rem}),
the monic polynomials $\text{p}_n(\xi;a,b)$ \eqref{cHp} and  $\text{p}_n(\xi)$ \eqref{pnA}  thus coincide, as solutions of the eigenvalue equation \eqref{DE-cH} corresponding to the same nondegenerate   eigenvalue $E_n$.  
The upshot is that $\xi_j^{(n)}=\xi_j$ ($j=1,\ldots ,n$), so the asserted inequalities for the zeros are again inherited from those  for  $\boldsymbol{\xi}=\boldsymbol{\xi}_{\text{A},\tilde\rho}^{(r)}$ stemming from Proposition \ref{alpha=0-Bethe-roots-A:prp}.
\end{proof}

\begin{remark}\label{lower-bound-cH:rem}
Because the weight function $\Delta (\xi) $ \eqref{weight-cH}  is even in $\xi$, the symmetric continuous Hahn polynomial  $\emph{p}_n(\xi;a,b)$ \eqref{cHp} is even or odd in $\xi$ depending whether $n$ is
even or odd, respectively. Its zeros are therefore positioned symmetrically around the origin:
$\xi^{(n)}_{n+1-j}=-\xi^{(n)}_j$ ($j=1,\ldots ,n$).  For $a,b>0$, the following lower bounds for the positive zeros of $\emph{p}_n(\xi;a,b)$ \eqref{cHp} are now an immediate consequence of Theorem \ref{cH:thm} (upon picking $j^\prime=n+1-j$):
\begin{equation}
\frac{\pi(n+1-2j)}{2 k^{(n)}_-(a,b)} \le \xi^{(n)}_j \qquad\text{for}\  j=1,\ldots ,\lfloor {\textstyle \frac{n}{2}} \rfloor .
\end{equation}
\end{remark}

\begin{remark}
It follows from the proofs of Theorems \ref{W:thm}, \ref{AW:thm} and \ref{cH:thm} that the vector of decreasingly ordered zeros $\boldsymbol{\xi}^{(n)}:=(\xi_1^{(n)},\ldots ,\xi_n^{(n)})$ of the Wilson, Askey-Wilson and symmetric continuous Hahn polynomials minimize $V^{(r)}_{\text{B},\rho} (\boldsymbol{\xi}; 0,0)$ and $V^{(t)}_{\text{B},\rho} (\boldsymbol{\xi}; 0,0)$  for $K=4$, $L=1$,
and 
$V^{(r)}_{\text{A},\tilde\rho} (\boldsymbol{\xi}; 0,\beta_n)$ for $K=2$, $L=1$, where the parameters are taken from Eqs.  \eqref{wilson-parameters}, \eqref{askey-wilson-parameters} and \eqref{cH-parameters}, respectively. A different interpretation for these zeros as the minimizers of corresponding $n$-particle Ruijsenaars-Schneider type Hamiltonians from Refs. \cite{die:integrability,die:difference} has been pointed out in Refs. 
\cite{oda-sas:equilibria,die:equilibrium,oda-sas:equilibrium}.
\end{remark}

\begin{remark}\label{parameter-extension:rem}
In view of Remark \ref{complex-parameters:rem}, the inequalities in Theorem \ref{AW:thm}  persist in the case that $a,b,c,d$  become complex within the open unit disc with non-real parameters arising in complex conjugate pairs. 
Similarly, Theorem \ref{W:thm}  and Theorem \ref{cH:thm} extend to the situation that $\text{Re}(a), \text{Re}(b),\text{Re}(c),\text{Re}(d)>0$ and $\text{Re}(a), \text{Re}(b)>0$
with non-real parameters occurring in complex conjugate pairs, upon replacing
$k^{(n)}_-(a,b,c,d)$ by $k^{(n)}_-\bigr(\text{Re}(a),\text{Re}(b),\text{Re}(c),\text{Re}(d)\bigr)$ and
$k^{(n)}_-(a,b)$ by $k^{(n)}_-\bigr(\text{Re}(a),\text{Re}(b)\bigr)$, respectively.
\end{remark}

\begin{remark}
The following small numerical sample illustrates that the bounds stemming from Theorem \ref{AW:thm}
estimate the positions of the first few roots of the Askey-Wilson polynomials reasonably well for parameter values
$a,b,c,d,q$ not far from $0$.  
Notice, however, that for the large(st) root(s) our upper bound soon becomes trivial ($>\pi$)  even in a small example.
The lower bounds for the roots of the Wilson polynomials and of the continuous Hahn polynomials originating from Theorem \ref{W:thm} and from Theorem \ref{cH:thm} (cf. also Remark \ref{lower-bound-cH:rem})  tend to be less sharp than in the Askey-Wilson case. 

\begin{table}[]
\centering
\caption{Roots and their bounds for the Askey-Wilson polynomial $\text{p}_n(\xi ;a,b,c,d;q)$ \eqref{AWp} with $n=5$, $a=0.300$, $b=-0.200$, $c=0.150$, $d=0.100$, and  $q=0.100$.}
\label{AW:table}
\begin{tabular}{@{}|lc|ccccc|@{}}
\midrule 
                                     &    $n=5$  & $\xi^{(n)}_5$ & $\xi^{(n)}_4$ &  $\xi^{(n)}_3$ & $\xi^{(n)}_2$&       $\xi^{(n)}_1$           \\ \midrule
 & Root & 0.496                      &  0.997                      & 1.508                      & 2.033  & 2.577    \\ \cmidrule(r){1-2}
                & Lower bound  & 0.400                       & 0.800                      & 1.200                      & 1.600  & 2.000  \\ \cmidrule(r){1-2} 
 &Upper bound  & 0.675                       & 1.350                     & 2.025                      & 2.700  & 3.375 \\ 
   \midrule
  \end{tabular}
\end{table}

\begin{table}[]
\centering
\caption{Roots and their lower bounds for the Wilson polynomial $\text{p}_n(\xi ;a,b,c,d)$ \eqref{Wp} with $n=5$,   $a=1.150$, $b=1.100$, $c=1.000$,  and $d=0.900$.}
\label{W:table}
\begin{tabular}{@{}|lc|ccccc|@{}}
\midrule 
                                     &    $n=5$  & $\xi^{(n)}_5$ & $\xi^{(n)}_4$ &  $\xi^{(n)}_3$ & $\xi^{(n)}_2$&       $\xi^{(n)}_1$           \\ \midrule
 & Root & 0.632                       & 1.292                      & 2.090                      & 3.099 & 4.477      \\ \cmidrule(r){1-2} 
 &Lower bound  & 0.264                      & 0.528                      & 0.793                      & 1.057  & 1.321 \\
   \midrule
  \end{tabular}
\end{table}

\begin{table}[]
\centering
\caption{Positive roots and their lower bounds for the symmetric continuous Hahn polynomial $\text{p}_{n}(\xi ;a,b)$ \eqref{cHp} with $n=10$, $a=1.100$ and $b=0.900$.}
\label{cH:table}
\begin{tabular}{@{}|lc|ccccc|@{}}
\midrule 
                                     &    $n=10$  & $\xi^{(n)}_5$ & $\xi^{(n)}_4$ &  $\xi^{(n)}_3$ & $\xi^{(n)}_2$&       $\xi^{(n)}_1$           \\ \midrule
 & Root & 0.261                       & 0.838                      & 1.554                     & 2.481 & 3.770     \\ \cmidrule(r){1-2} 
 &Lower bound  & 0.131                      & 0.392                     &0.653                      & 0.915  & 1.176 \\
   \midrule
  \end{tabular}
\end{table}
\end{remark}

\section*{Acknowledgments}
Thanks are due to the referees for pointing out some improvements of the presentation.

\bibliographystyle{amsplain}

\end{document}